\documentclass{birkau}

\usepackage{amsmath,amssymb,url}

\numberwithin{equation}{section}

\theoremstyle{plain}
\newtheorem{theorem}{Theorem}[section]
\newtheorem{lemma}[theorem]{Lemma}
\newtheorem{proposition}[theorem]{Proposition}
\newtheorem{corollary}[theorem]{Corollary}

\theoremstyle{definition}
\newtheorem{definition}[theorem]{Definition}

\newtheorem{remark}[theorem]{Remark}
\newtheorem{example}[theorem]{Example}
\newtheorem*{KM}{Krein-Milman theorem}

\usepackage{enumitem}
\usepackage{mathtools}
\usepackage{xfrac}
\usepackage{tikz,pgfplots}
\usepackage{tikz-cd}
\usepackage{mathrsfs}
\usepackage{bbm}
\usetikzlibrary{patterns}

\DeclareMathOperator{\N}{\mathbb{N}}

\DeclareMathOperator{\Q}{\mathbb{Q}}
\DeclareMathOperator{\R}{\mathbb{R}}

\DeclareMathOperator{\St}{St}

\DeclareMathOperator{\Aff}{Aff} 

\newcommand{\calP}{\mathcal{P}}
\newcommand{\calL}{\mathcal{L}}

\newcommand{\frakm}{\mathfrak{m}}
\newcommand{\frakn}{\mathfrak{n}}

\newcommand{\ba}{\mathbf{a}}
\newcommand{\bb}{\mathbf{b}}

\renewcommand{\epsilon}{\varepsilon}
\providecommand{\abs}[1]{\lvert#1\rvert}
\providecommand{\norm}[1]{\lVert#1\rVert}

\DeclareMathOperator{\Rad}{Rad}
\DeclareMathOperator{\Max}{Max}
\DeclareMathOperator{\ext}{ext}
\DeclareMathOperator{\id}{id}
\DeclareMathOperator{\conv}{co}

\begin{document}

\title[Ordered group-valued probability]{Ordered group-valued probability, positive operators, and integral representations}

\author[T. Kroupa]{Tom\'a\v{s} Kroupa}
\address{The Czech Academy of Sciences\\ Institute of Information Theory and Automation\\\ Prague, Czech Republic}
\urladdr{http://staff.utia.cas.cz/kroupa/}
 \email{kroupa@utia.cas.cz}   

\thanks{The work on this paper has been supported from the GA\v{C}R grant project GA17-04630S}


\begin{abstract}
Probability maps are additive and normalised maps taking values in the  unit interval of a  lattice ordered Abelian group. They appear in  theory of affine representations and they are also a semantic counterpart of H\'ajek's probability logic. In this paper we obtain a~correspondence between probability maps and positive operators of~certain Riesz spaces, which extends the well-known representation theorem of~real-valued MV-algebraic states by positive linear functionals. When the codomain algebra contains all continuous functions, the set of~all probability maps is convex, and we prove that its extreme points coincide with homomorphisms. We also show that probability maps can be viewed as a collection of states indexed by maximal ideals of a codomain algebra, and we characterise this collection in special cases.
\end{abstract}

\maketitle

\section{Introduction}\label{sec:Intro}
States are real-valued functions on MV-algebras (or, equivalently, unital lattice ordered Abelian groups), which are counterparts of real-valued probability measures of Boolean algebras. See \cite{FlaminioKroupa15,Mundici11} for a detailed treatment of~states and their relation to ordered algebras, mathematical logic, and probability theory. In this paper we make an attempt at a systematic study of \emph{probability maps}, that is, states taking values in more general ordered groups than $\R$, with regard to their representation by positive operators of certain Riesz spaces. The achieved results generalise important theorems about real-valued states, such as their correspondence to positive linear functionals and  the integral representation theorem. When compared with \cite{Boccuto2017}, we allow for a~more general codomain of probability maps than Riesz MV-algebras. The~assumption of Dedekind completeness is frequently adopted when working with the~codomain of states or positive operators. However, in this paper it is used 
 only in Theorem \ref{thm:intvecmeas}.

The research on states is related to a number of interesting topics in~algebraic logic, ordered groups, and probability on MV-algebras. Let us mention just a few of them as a motivation for this paper and as a way of recalling previous results obtained in this area. Our list is by no means exhaustive.
\begin{description}
\item[\emph{Generalised states}] In an effort to develop an algebraic semantics (in the~ varietal sense) for H\'ajek's probability logic \cite[Chapter 8.4]{Hajek98}, the author and Marra introduced the notion of generalised state and investigated the associated two-sorted variety of algebras \cite{KroupaMarraSC17}. This approach opens space for an equational treatment of probability and makes it possible to apply universal constructions in this field.
\item[\emph{Affine representations}] The classical theory of affine representations is used to classify certain types of ordered Abelian groups; see \cite{Goodearl86}. The natural evaluation map from Example \ref{ex:affine}, which sends the universe of~algebra into the set of affine continuous functions on the state space, is a~prototypical example of a~probability map since it is additive and positive, but it does not preserve lattice operations. Affine representations of~MV-algebras are investigated in~\cite{Boccuto2017}.
\item[\emph{Non-archimedean probability}] Probability measures with co\-do\-mains more general than $[0,1]$  can be considered naturally; see \cite[Example~1]{KroupaMarraSC17}. It is well-known that the set of all positive integers $\N$ admits no uniform probability distribution in the classical sense. On the other hand, there is a uniform probability map on the finite-cofinite Boolean algebra~$\mathcal{B}$ over $\N$. Assume the codomain of a probability map is Chang's algebra $C\coloneqq \{0,\varepsilon,2\varepsilon,\dots,1-2\varepsilon,1-\varepsilon,1\}$, where the symbol $\varepsilon$ represents the smallest infinitesimal greater than $0$.  Then  $p\colon \mathcal{B}\to C$ defined as $p(A)\coloneqq \abs{A}\varepsilon$, when $A$ is finite, and $p(A)\coloneqq 1-\abs{\bar{A}}\varepsilon$, if $A$ is cofinite, is a probability map. It can be thought of as a ``uniform distribution'' over $\N$ since $p(\{n\})=\varepsilon$ for each $n\in \N$.
\end{description}

The mathematical tools used in this paper are not only algebraic, but they also involve  convexity and infinite-dimensional geometry, which is dictated by the rich structure we are dealing with (spaces of~positive operators). This wide arsenal of techniques makes it possible to draw a parallel between certain ordered algebras and their positive maps, which is witnessed by some of the main results of this paper such as Theorem \ref{thm:extoperator}.

\begin{table}[htp]
\begin{center}
\begin{tabular}{ll}
\textbf{Structure} & \textbf{Map}  \\ \hline
MV-algebra & Probability map \\
Unital Abelian $\ell$-group & Unital positive group homomorphism \\
Riesz space $C_{\R}(X)$ & Unital positive linear operator
\end{tabular}
\end{center}
\label{Tab}
\caption{}
\end{table}%
\noindent
It is well-known that the class of MV-algebras is equational, whereas unital Abelian $\ell$-groups do not form a variety of algebras. Interestingly enough, the~positivity and the additivity conditions of group homomorphisms take form of two simple equational laws in case of probability maps between MV-algebras; cf. Definition \ref{def:genstate} and Lemma \ref{lem:extgroup}.

The article is structured as follows. We fix our notation and terminology in~Section \ref{sec:Prelim}. In particular, we repeat basic facts about MV-algebras and~positive operators between Riesz spaces, which form a necessary background for dealing with real-valued states of MV-algebras and their generalisations. Throughout the paper we use the apparatus of compact convex sets and affine continuous maps; the reader is referred to Appendix \ref{sec:coco} for summary. Our main results are formulated in Section \ref{sec:PM}.
\begin{itemize}
\item Theorem \ref{thm:extoperator} says that there is a one-to-one correspondence between probability maps and positive operators.
\item Corollary \ref{thm:extremestate} identifies extreme points of the convex set of all probability maps when the codomain is the algebra of all continuous functions.
\item Theorem \ref{thm:intvecmeas} provides sufficient conditions for an integral representation of a~probability map by a vector measure.
\end{itemize}
In Section \ref{sec:Adj} we formulate a framework for studying dual probability maps, which can be fully described in certain simple cases. The paper concludes with an outlook towards further research (Section \ref{sec:conc}).

\section{Preliminaries}\label{sec:Prelim}
First, we recall basic definitions and facts about MV-algebras \cite{CignoliOttavianoMundici00,Mundici11}. Then we discuss positive operators between Riesz spaces of continuous functions, which is instrumental in parts of theory of MV-algebraic states \cite{Mundici95,FlaminioKroupa15}.
\subsection{MV-algebras}
We follow the notation and terminology of \cite{CignoliOttavianoMundici00}, to which we refer the reader for background. Given an MV-algebra $(M,\oplus,\neg,0)$, we define as usual the~binary
operation $\odot$ and the constant $1$ together with the lattice
supremum $\vee$ and infimum~$\wedge$ on $M$. The lattice order of $M$ is
denoted by $\leq$.
\begin{example}[Standard MV-algebra]
Let $[0,1]$ be the real unit interval equipped with the operations $a\oplus b \coloneqq \min(a+b,1)$, $\neg a \coloneqq 1-a$, and the~constant $0$. This structure is an MV-algebra and its order $\leq$ coincides with the~usual total order of the real unit interval $[0,1]$.
\end{example}

\begin{example}[MV-algebra of continuous functions]\label{ex:MVcont}
Let $X$ be a compact Hausdorff topological space and $C(X)$ be the collection of all continuous functions $X\to [0,1]$. Since $\oplus$ and $\neg$ are continuous functions, $C(X)$ becomes an MV-algebra equipped with the operations of the product MV-algebra $[0,1]^{X}$. The MV-algebra $C(X)$ is \emph{separating}: For every $x,y\in X$ with $x\neq y$, there is an element $a\in C(X)$ such that $a(x)\neq a(y)$.
\end{example}

\noindent
By McNaughton theorem \cite{McNaughton51} the free MV-algebra over $n$ generators coincides with the algebra of certain continuous functions defined on the hypercube $[0,1]^{n}$.
\begin{example}[The free MV-algebra]\label{ex:FMV}
Let $n\in \N$. We call $a\colon [0,1]^{n}\to [0,1]$ a \emph{McNaughton function} if $a$ is continuous and piecewise-linear, where each linear piece has integer coefficients only. Let $F_{\mathrm{MV}}(n)$ be the MV-algebra of~all $n$-variable McNaughton functions, where $F_{\mathrm{MV}}(n)$ is considered as an MV-subalgebra of $C([0,1]^{n})$. Then $F_{\mathrm{MV}}(n)$ is precisely the free $n$-generated MV-algebra. Loosely speaking,  McNaughton functions $[0,1]^{n}\to [0,1]$ are many-valued generalisations of Boolean functions $\{0,1\}^{n}\to \{0,1\}$.
\end{example}

An \emph{ideal} of an MV-algebra $M$ is a subset $\mathfrak{i}$ of $M$ containing $0$, closed with respect to $\oplus$, and if $a\in \mathfrak{i}$, $b\in M$ and $b\leq a$, then $b \in \mathfrak{i}$.
The notions of~MV-homomorphism and quotient are defined as usual. By $\Max M$ we denote the nonempty compact Hausdorff topological space of maximal ideals $\frakm$ of $M$. For any MV-algebra $M$, let $\Rad M$ be the \emph{radical} of $M$,
\[
\Rad M \coloneqq \bigcap \Max M.
\]
Always $0\in \Rad M$. Nonzero elements of $\Rad M$ are called \emph{infinitesimals}. An~MV-algebra $M$ is said to be \emph{semisimple} if it has no infinitesimals, that is, $\Rad M=\{0\}$.

\begin{theorem}[{\cite[Theorem 4.16]{Mundici11}}]\label{thm:Holder}
Let $M$ be an MV-algebra. 
\begin{enumerate}
\item For any $\frakm\in \Max M$ there exist a unique MV-subalgebra $I_{\frakm}$ of $[0,1]$ together with a unique MV-isomorphism \[h_{\frakm}\colon A/ \frakm \to I_{\frakm}.\]
\item Let the map $a\in M \mapsto a^{*}\in [0,1]^{\Max M}$ be defined by 
\[
a^{*}(\frakm) \coloneqq h_{\frakm}(a/ \frakm), \quad \frakm \in \Max M.
\]
Then the map $^{*}$ is an MV-homomorphism of $M$ onto a separating MV-subalgebra $M^{*}$ of $C(\Max M)$. Moreover, the map $^{*}$ is an MV-isomorphism onto $M^{*}$ if, and only if, $M$ is semisimple.
\end{enumerate}
\end{theorem}
\noindent
The first part of Theorem \ref{thm:Holder} is an MV-algebraic variant of H\"older's theorem.
For any MV-algebra $M$, the quotient algebra $M/ \Rad M$ is semisimple, and~the maximal ideal spaces $\Max M$ and $\Max (M/ \Rad M)$ are homeomorphic; see \cite[Lemma 4.20]{Mundici11}.

Let $(G,u)$ be a unital lattice ordered Abelian group \cite{Goodearl86} (unital $\ell$-group, for short), where $u$ is an order unit. Put $\Gamma(G) \coloneqq [0,u]$, where $0$ is the neutral element of $G$ and $[0,u]$ denotes the corresponding order interval in $G$.  For all $a,b\in \Gamma(G)$ we define
\[
a\oplus b \coloneqq (a+b) \wedge u, \qquad \neg a \coloneqq u-a,
\]
where $+,\wedge$, and $-$ are the operations of $G$. Then $(\Gamma(G),\oplus,\neg,0)$ becomes an MV-algebra. 
Mundici's theory of $\Gamma$ functor \cite[Chapter 7]{CignoliOttavianoMundici00} says that $\Gamma(G)$ is  the most general example of an MV-algebra. Specifically, for every MV-algebra $M$ there is a unique unital $\ell$-group $\Xi(M)$ such that $M$ is isomorphic to $\Gamma(\Xi(M))$.

We define $a\ominus b\coloneqq a\odot \neg b$, for all $a,b\in M$. In Section \ref{sec:PM} we make an~ample use of several MV-algebraic identities:
\begin{enumerate}[label=(MV\arabic*),leftmargin=1.5\parindent]
\item \label{MV1} $b\wedge \neg a=b\ominus (a\odot b)$
\item \label{MV2}	$a\oplus b=a\oplus (b\wedge \neg a)$
\item\label{MV3}	$a\odot (b\wedge \neg a)=0$
\end{enumerate} 
In particular,  \ref{MV2} and \ref{MV3}  say that any element in $M$ of the form $a\oplus b$ can be additively cut into two disjoint parts, $a$ and $b\wedge \neg a$. 
In order to show that \ref{MV1} is true, recall that, by the definition of $\wedge$ in $M$, we have 
\[
b\wedge \neg a =b\odot (\neg b \oplus \neg a)=b\odot \neg (a\odot b)=b\ominus (a\odot b).
\]
 By \cite[Proposition 1.1.6]{CignoliOttavianoMundici00} the~ operation~$\oplus$ distributes over $\wedge$, 
\[
a\oplus (b\wedge \neg a) = (a\oplus b) \wedge (a\oplus \neg a)=a\oplus b, 
\]
which proves \ref{MV2}. Finally, \ref{MV3} holds,
\[
a\odot (b\wedge \neg a)=a\odot (b\ominus (a\odot b))=(a\odot b) \odot \neg (a\odot b)=0.
\]

\subsection{Positive operators}
The standard reference for Banach lattices of continuous functions and their positive operators is \cite{Schaefer74}. Positive operators between Riesz spaces are thoroughly studied also in \cite{AliprantisBurkinshaw06,Zaanen12}. 

Let $X$ and $Y$ be compact Hausdorff spaces. By $C_{\R}(X)$ we denote the~Banach lattice of all continuous functions $a\colon X\to \R$, endowed with the~ supremum norm~$\norm{.}$, where \[\norm{a}\coloneqq \sup\; \{\abs{a(x)} \mid x\in X\}, \qquad a\in C_{\R}(X).\] A \emph{unital positive linear operator} is a linear mapping $L\colon C_{\R}(X)\to C_{\R}(Y)$ that is \emph{positive} ($L(a)\geq 0$ for all $a\geq 0$) and \emph{unital} ($L(1)=1$). Any such $L$ is necessarily a \emph{monotone} mapping ($a\leq b$ implies $L(a)\leq L(b)$) and a~\emph{continuous} linear operator (there is some $C>0$ such that $\norm{L(a)}\leq C\norm{a}$ for all $a\in C_{\R}(X)$). Let 
\(
\calL^{+}(C_{\R}(X),C_{\R}(Y))
\)
be the set of all unital positive linear operators $C_{\R}(X)\to C_{\R}(Y)$. Then \(
\calL^{+}(C_{\R}(X),C_{\R}(Y))
\)
is clearly a convex subset of the linear space of all maps $C_{\R}(X)\to C_{\R}(Y)$. Ellis-Phelps theorem (see \cite{Ellis64,Phelps63}) then describes the extreme boundary of $\calL^{+}(C_{\R}(X),C_{\R}(Y))$. We use a formulation of this result from \cite[Theorem III.9.1]{Schaefer74}.

\begin{theorem}\label{thm:extremeoperator}
Let $L\in \calL^{+}(C_{\R}(X),C_{\R}(Y))$. Then the following are equivalent:
\begin{enumerate}
\item $L$  is an extreme point of $\calL^{+}(C_{\R}(X),C_{\R}(Y))$.
\item $L$ is a lattice homomorphism.
\item There is a unique continuous map $f\colon Y\to X$ such that $L(a)=a\circ f$, for all $a\in C_{\R}(X)$.
\end{enumerate}
\end{theorem}

If $Y=\{y\}$, then $C_{\R}(Y)$ and $\R$ are identified, and any $F \in \calL^{+}(C_{\R}(X),\R)$ is called a \emph{unital positive linear functional}. The~ shorter notation $\calL^{+}(C_{\R}(X))$ will be used in place of $\calL^{+}(C_{\R}(X),\R)$. The space $\calL^{+}(C_{\R}(X))$ is endowed with the subspace weak$^{*}$-topology that is inherited from the dual space $C^{*}_{\R}(X)$ of~$C_{\R}(X)$. This makes $\calL^{+}(C_{\R}(X))$ into a compact convex subset of the locally convex Hausdorff space $C^{*}_{\R}(X)$. Let $\Delta(X)$ be the compact convex set of all regular Borel probability measures over $X$ (see Example \ref{ex:PMs}). Given a~probability measure $\mu\in \Delta(X)$, define
\[
F_{\mu}(a) \coloneqq \int_{X} a\; \mathrm{d}\mu, \qquad a\in C_{\R}(X),
\]
and observe that $F_{\mu}\in \calL^{+}(C_{\R}(X))$. Then Riesz theorem  can be formulated as follows. 
\begin{theorem}
The map $\Delta(X)\to \calL^{+}(C_{\R}(X))$ given by 
\[
\mu \mapsto F_{\mu}
\]
is an affine homeomorphism.
\end{theorem}

\subsection{States}
Let $M$ be an MV-algebra. A \emph{state of $M$} is a real function $s\colon M\to [0,1]$ satisfying $s(1)=1$ and the following condition for all $a,b\in M$:
\begin{equation}\label{eq:add}
\text{If $a\odot b=0$, then $s(a\oplus b)=s(a)+s(b)$.}
\end{equation}
Every MV-algebra carries at least one state as a consequence of Theorem~\ref{thm:Holder} and the fact that $\Max M\neq \emptyset$. Namely take an arbitrary $\mathfrak{m}\in \Max M$, put 
\begin{equation}\label{def:statemorph}
s_{\mathfrak{m}}(a)\coloneqq a^{*}(\mathfrak{m}), \quad a\in M,
\end{equation}
and observe that the MV-homomorphism $s_{\mathfrak{m}}\colon M\to [0,1]$ is a state of $M$. In fact every MV-homomorphism $h\colon M\to [0,1]$ is necessarily of the form $h=s_{\mathfrak{m}}$, for some $\mathfrak{m}\in \Max M$.

States of $M$ are in one-one correspondence with states of the enveloping $\ell$-group of $M$. Specifically, a \emph{state of $G$}, where $(G,u)$ is unital $\ell$-group, is~a~unital positive real homomorphism of $G$, that is, a group homomorphism $s\colon G\to \R$ such that $s[G^{+}]\subseteq \R^{+}$ and $s(u)=1$. See \cite[Chapter 4]{Goodearl86} for~a~thorough exposition of $\ell$-group states. 
\begin{theorem}\label{thm:groupext}
Let $M$ be an MV-algebra and $\Xi(M)$ be~its enveloping $\ell$-group. Then the restriction of a state $s$ of $\Xi(M)$ to a function $s\upharpoonright M\to [0,1]$ determines a~bijection between the states of $\Xi(M)$ and the~states of $M$.
\end{theorem}
\begin{proof}
 See \cite[Theorem 2.4]{Mundici95}. 
\end{proof}
By $\R^{M}$ we denote the locally convex Hausdorff space of real functions over $M$ equipped with the product topology. The \emph{state space of $M$} is a~nonempty set
\begin{equation}\label{def:St}
\St M \coloneqq \{s \mid \text{$s\colon M\to [0,1]$ is a state}\}.
\end{equation}
We will use the convergence of nets in topological spaces \cite[Definition 11.2]{Willard70}.
A net of states $(s_{\gamma})$ converges to a state $s$ in $\St M$ if it converges pointwise,
\begin{equation}\label{def:convstates}
s_{\gamma}(a) \to s(a)\quad \text{for every $a\in M$.}
\end{equation}
A straightforward verification proves that  $\St M$ is a compact convex subset of~the Tychonoff cube $[0,1]^{M}\subseteq \R^{M}$. By Krein-Milman theorem we know that $\St M$ coincides with the closed convex hull of the extreme boundary $\ext \St M$, which is precisely the set of all MV-homomorphisms into $[0,1]$,
\begin{equation}\label{eq:stspacemax}
\ext \St M = \{s_{\mathfrak{m}} \mid \mathfrak{m}\in \Max M\}.
\end{equation}
Since the map $\mathfrak{m}\in \Max M \mapsto s_{\mathfrak{m}}\in \St M$ is bijective and continuous, the~extreme boundary $\ext \St M$ is compact.

From the perspective of states, infinitesimal elements are irrelevant. This fact is in sharp contrast with behavior of states on pseudo MV-algebras \cite{diNolaDvurJak04}.
\begin{proposition}[{\cite[Proposition 3.1]{Mundici95}}]\label{pro:maxsemquo}
For any MV-algebra $M$, the state spaces $\St M$ and $\St (M/\Rad M)$ are affinely homeomorphic.
\end{proposition}
\noindent
In the rest of this section we will always assume that an MV-algebra $M$ is~ semisimple. If it is not the case, we can simply replace $M$ with $M/ \Rad M$ in~the~ arguments involving the corresponding state spaces.

 The main characterisation of states is based on integral representation of certain linear functionals. The reader is invited to consult Appendix \ref{sec:coco} for~all the concepts related to compact convex sets. Let $\Delta(\Max M)$ be the Bauer simplex of all regular Borel probability measures on $\Max M$ and $\calL^{+}(C_{\R}(\Max M))$ be the Bauer simplex of all unital positive linear functionals on $C_{\R}(\Max M)$. For any $\mu \in \Delta(\Max M)$, observe that the function $s_{\mu}\colon M\to [0,1]$ defined by
\begin{equation}\label{thm:int}
s_{\mu}(a)\coloneqq \int\limits_{\Max M} a^{*} \;\mathrm{d}\mu, \qquad a\in M,
\end{equation}
is a state of $M$. Note that the restriction of any  $L \in \calL^{+}(C_{\R}(\Max M))$ to $M^{*}$ is also a state.
\begin{theorem}[{\cite{Panti08_InvMeasures,FlaminioKroupa15}}]\label{thm:intrep}
Let $M$ be a semisimple MV-algebra. Then:
\begin{enumerate}
\item The restriction map
\begin{equation}
L\in \calL^{+}(C_{\R}(\Max M)) \mapsto L  \upharpoonright M^{*}\in \St M^{*}
\end{equation}
is an affine homeomorphism of $\calL^{+}(C_{\R}(\Max M))$ onto $\St M^{*}$.
\item The map 
\begin{equation}
\mu \in \Delta(\Max M) \mapsto s_{\mu}\in \St M
\end{equation}
is an affine homeomorphism of $\St M$ onto $\Delta(\Max M)$.
\end{enumerate}
\end{theorem}
\noindent
The difficult part of the theorem says that states of a (semisimple) MV-algebra are uniquely extendible to positive linear functionals of the ambient Banach lattice of continuous functions. Then it is a consequence of Riesz theorem that any state $s\in \St M$ is of the form $s=s_{\mu}$, for a unique measure $\mu \in \Delta(\Max M)$.

 Summing up, the state space $\St M$ of any MV-algebra $M$ is a Bauer simplex that is affinely homeomorphic to any of the Bauer simplices appearing on the~list below:
\begin{itemize}
\item The state space of the maximal semisimple quotient $M/\Rad M$.
\item The state space of the enveloping $\ell$-group $\Xi(M)$.
\item Unital positive linear functionals $\calL^{+}(C_{\R}(\Max M))$.
\item Regular Borel probability measures $\Delta(\Max M)$.
\end{itemize}

\section{Probability maps} \label{sec:PM}
 A probability map (Definition \ref{def:genstate}) is an additive normalised map of an MV-algebra~$M$ into a more general codomain MV-algebra $N$ than the standard algebra $[0,1]$. At the same time, Definition \ref{def:genstate} extends the concept of ``generalised state" introduced by the author and Marra in \cite{KroupaMarraSC17}, where the domain MV-algebra $M$ is taken to be a Boolean algebra. Unlike other presentations of~probability-like maps of MV-algebras or similar structures, the one in Definition \ref{def:genstate} is purely equational. Nevertheless, Proposition \ref{pro:PMequiv} below shows that  probability maps admit also a more traditional definition that is analogous to the defining condition \eqref{eq:add} of $[0,1]$-valued states.

Throughout this section, $M$ and $N$ denote MV-algebras.
 Same symbols are used to denote the operations and the constants in both $M$ and $N$; their meaning will always be clear from the context.
\begin{definition}\label{def:genstate}
A \emph{probability map} is a function $p\colon M\to N$ satisfying the~following identities for all $a,b\in M$.
\begin{enumerate}[label=(P\arabic*),leftmargin=1.2\parindent]
\item \label{P1} $p(a\oplus b)=p(a)\oplus p(b\wedge \neg a)$
\item \label{P2} $p(\neg a)=\neg p(a)$
\item \label{P3} $p(1)=1$
\end{enumerate}
\end{definition}
It is not difficult to show that the axioms \ref{P1}--\ref{P3} are independent.
Note that $p$ is also normalised at $0$ since, by \ref{P2} and \ref{P3}, $p(0)=\neg p(1 )=0$.
Every MV-homomorphism is a probability map as \ref{P1} follows from \ref{MV2} in~this case. The axiomatisation in Definition \ref{def:genstate} is inspired by that of internal states of MV-algebras studied by Flaminio and Montagna \cite{FlaminioMontagna09}. An \emph{internal state} of an MV-algebra $M$ is a mapping $\sigma\colon M\to M$ satisfying the conditions (1)--(4) for all $a,b\in M$:
\begin{enumerate}
\item $\sigma(0)=0$
\item $\sigma(\neg a)=\neg \sigma(a)$
\item $\sigma(a\oplus b)=\sigma(a) \oplus \sigma(b\ominus (a\odot b))$
\item $\sigma(\sigma(a)\oplus \sigma(b))=\sigma(a)\oplus \sigma(b)$
\end{enumerate}
By \ref{MV1} we can rewrite (P1) as
		\begin{enumerate}[label=(P\arabic*'),leftmargin=1.3\parindent]
			\item\label{P1'} $p(a\oplus b)=p(a) \oplus p(b\ominus (a\odot b))$
\end{enumerate}
which is among the defining conditions of internal states. Thus, we can equivalently define a probability map to be a function $p\colon M\to N$ satisfying \ref{P1'}, (P2), and (P3). This also shows that every internal state of $M$ is a probability map $M\to M$. However, the converse is typically not true. Whereas any internal state is an idempotent map,  there are already MV-endomorphisms that fail to be idempotent.

We will collect useful elementary properties of probability maps into this lemma.

\begin{lemma}\label{pro:basic}
	Let $p\colon M\to N$ be a probability map. For all $a,b\in M$:
	\begin{enumerate}
		\item If $a\leq b$, then $p(a)\leq p(b)$.
		\item $p(a\oplus b)\leq p(a)\oplus p(b)$, and if $a\odot b=0$, then $p(a\oplus b)=p(a) \oplus p(b)$.
		\item $p(a\ominus b)\geq p(a)\ominus p(b)$, and if $b\leq a$, then $p(a\ominus b)= p(a)\ominus p(b)$.
		\item $p(a\odot b)\geq p(a)\odot p(b)$, and if $a\odot b=0$, then $p(a)\odot p(b)=0$.
	\end{enumerate}	
\end{lemma}
\begin{proof}
	(1)  Let $a\leq b$. Then $b=a\oplus (b\ominus a)$ and using \ref{P1} we get
	\[
	p(b)=p(a\oplus (b\ominus a))=p(a)\oplus p((b\ominus a)\wedge \neg a) \geq p(a).
	\]
	
	(2)\ It follows from \ref{P1'} and (1) that 
	\[
	p(a\oplus b)=p(a) \oplus p(b\ominus (a\odot b)) \leq p(a) \oplus p(b).
	\]
	Let $a\odot b=0$. Then 
	\(
	p(a\oplus b)=p(a) \oplus p(b\ominus 0)=p(a)\oplus p(b)
	\).

	(3)\ By the definition of $\ominus$, \ref{P2}, and (2),
	\begin{align*}
	p(a\ominus b) &=p(\neg ( \neg a\oplus b))=\neg p(\neg a\oplus b)\\&\geq \neg (\neg p(a)\oplus p(b))=p(a)\odot \neg p(b)=p(a)\ominus p(b).
	\end{align*}
	Let $b\leq a$. Then $\neg a\odot b=0$ and the equality in the  formula above follows from (2).

	(4) We get
	\begin{equation*}
	p(a\odot b)=p(a\ominus \neg b)\geq p(a)\ominus p(\neg b)=p(a)\odot \neg p(\neg b)=p(a)\odot p(b).
	\end{equation*}
	Let $a\odot b=0$. Then 
	\(
	p(a)\odot p(b) \leq p(a\odot b)=p(0)=0\).
\end{proof}
Probability maps can be viewed as genuinely additive maps between MV-algebras.
	
\begin{proposition}\label{pro:PMequiv}
		Let $p\colon M\to N$ satisfy $p(0)=0$ and $p(1)=1$. The following are equivalent.
		\begin{enumerate}
			\item $p$ is a probability map.
			\item $p(a\oplus b)=p(a)+p(b)-p(a\odot b)$, for every $a,b \in M$, where $+$ and $-$ are operations in the enveloping unital $\ell$-group $\Xi(N)$ of $N$.
			\item If $a\odot b=0$, then $p(a\oplus b)=p(a) \oplus p(b)$ and $p(a)\odot p(b)=0$, for all $a,b\in M$.
			\item If $a\odot b=0$, then $p(a\oplus b)=p(a) + p(b)$, for all $a,b\in M$. 			
		\end{enumerate}
	\end{proposition}
	\begin{proof}
 (1) $\Rightarrow$ (2)  By Lemma \ref{pro:basic}, \ref{MV3}, and the fact that 
 \[
 c+ d=(c\oplus d)+(c\odot d)
 \]
  for all $c,d\in M$ (see \cite[Lemma 2.1.3]{CignoliOttavianoMundici00}), we obtain
		\begin{align*}
			p(a\oplus b) =p(a) \oplus (p(b)\ominus p(a\odot b))=p(a) + (p(b) - p(a\odot b)).
		\end{align*} 
		
 (2) $\Rightarrow$ (3) Let $a\odot b=0$. Then $p(a)+p(b) =p(a\oplus b)\leq 1$, which means that $p(a\oplus b)=p(a)\oplus p(b)$ and $p(a)\odot p(b)=0$.

 (3) $\Rightarrow$ (1) The identity $a\odot  \neg a=0$ and assumption yield $p(a)\odot p(\neg a)=0$. Further,
		\[
		1=p(1)=p(a\oplus \neg a)=p(a)\oplus p( \neg a), \quad \text{for all $a \in M$.}
		\]
		Therefore $p(\neg a)=\neg p(a)$ by \cite[Lemma 1.1.3]{CignoliOttavianoMundici00} and thus \ref{P2} in Definition \ref{def:genstate} holds. Finally, we  prove \ref{P1}. Using \ref{MV1}--\ref{MV3}, we obtain
		\[
		p(a\oplus b)=p(a\oplus (b\wedge \neg a ))=p(a) \oplus p(b\wedge \neg a ), \quad a,b \in M.
		\]
		Hence, $p$ is a probability map. 
		
The conditions (3) and (4) are equivalent since $c\odot d=0$ holds in $N$ if, and only if, $c+d\leq 1$ in the enveloping $\ell$-group $\Xi(N)$ of $N$.		
\end{proof}

Put
\[
\calP(M,N) \coloneqq \{p \mid \text{$p\colon M\to N$ is a probability map}\}.
\]
 Using
 the defining condition \eqref{eq:add} of real-valued states and Proposition \ref{pro:PMequiv}(4), one can immediately see that states of MV-algebras are exactly the probability maps whose codomain $N$ is the standard MV-algebra $[0,1]$. This means that $\calP(M,[0,1])$ is the same as the classical state space \eqref{def:St} of $M$,
\[
\calP(M,[0,1]) = \St M.
\]
The state space $\St M$ is always nonempty for any $M$, whereas it can easily happen $\calP(M,N)=\emptyset$. For example, let $M\coloneqq [0,1]$ and $N$ be the two-element Boolean algebra $\{0,1\}$. Any $p\in \calP([0,1],\{0,1\})$ must  necessarily be a state $p\colon [0,1]\to [0,1]$ whose range is $\{0,1\}$. But there is no such $s$ since the only existing state of $[0,1]$ is the identity map. Hence, $\calP([0,1],\{0,1\})=\emptyset$.

In general, even when the set $\calP(M,N)$ is nonempty, it fails to be convex.
\begin{example}
In this example we put $M=N=F_{\mathrm{MV}}(1)$, where $F_{\mathrm{MV}}(1)$ is as in Example \ref{ex:FMV}. Then $\calP(M,M)\neq \emptyset$ since any MV-endomorphism of $M$ is a~probability map. In addition, there are probability maps $p\colon M\to M$ that are not MV-endomorphisms. For example, consider the probability map defined as
\begin{equation}\label{ex:PM}
p(a)(x) \coloneqq  x\cdot a(0) + (1-x)\cdot a(1), 
\end{equation}
for all $a\in M$ and every  $x\in [0,1]$. Then $p\in \calP(M,M)$ is not an MV-endomorphism since it fails to preserve suprema. Indeed, take $a\coloneqq \id$, where $\id$ is the identity function on $[0,1]$: 
\[
p(\id \vee \neg \id)=1 \neq \id \vee \neg \id = p(\neg \id) \vee p(\id).\]
With $\calP(M,M)$ being a subset of the real linear space $C_{\R}([0,1])^{M}$ of all maps $M\to C_{\R}([0,1])$, it is sensible to ask if $\calP(M,M)$  is a convex set. This is clearly not the case. It is enough to consider MV-endomorphisms $p_{1}$ and $p_{2}$ of $M$ defined by
\[
p_{1}(a)\coloneqq a, \quad p_{2}(a)(x) \coloneqq a(1-x), \quad a\in M, x\in [0,1].
\]
Observe that the range of the map $q\coloneqq \tfrac 12 (p_{1}+p_{2})$ is not even a subset of $M$ and so $q\notin \calP(M,M)$.
\end{example}
The previous example suggests that convexity of $\calP(M,N)$ might be attained   when the universe of algebra $N$ is a convex set in some linear space. This occurs, for example, when the codomain $N$ is a Riesz MV-algebra \cite{DiNolaLestean14}. However, such a level of generality is not needed in the paper, and we will always consider a convex codomain of a very special kind, the MV-algebra $C(X)$ introduced in Example~\ref{ex:MVcont}. Then the next result is immediate.
\begin{proposition}\label{pro:convex}
Let $M$ be an MV-algebra and $Y$ be a compact Hausdorff space. Then $\calP(M,C(Y))$ is a nonempty convex subset of the real linear space $C_{\R}(Y)^{M}$.
\end{proposition}
\begin{proof}
By Theorem \ref{thm:Holder} an MV-algebra $M$ has at least one MV-homomorphism  $h\colon M\to [0,1]$. Define $p_{h}\colon M\to C(Y)$ as
\begin{equation}\label{eq:proeq}
p_{h}(a)(y) \coloneqq h(a), \qquad a\in M, \enskip y\in Y.
\end{equation}
Then each $p_{h}(a)$ is a constant function. It is easy to see that $p_{h}$ is an MV-homomorphism. Thus, $\calP(M,C(Y))\neq \emptyset$.

Consider $p,q\in \calP(M,C(Y))$ and $\alpha \in [0,1]$.
 We need to check that the map $r\colon M\to C_{\R}(Y)$ defined by
\(
r\coloneqq \alpha p + (1-\alpha)q
\)
belongs to $\calP(M,C(Y))$. Since $\alpha a + (1-\alpha)b \in C(Y)$, for all $a,b\in C(Y)$, the map $r$ ranges in $C(Y)$. Clearly, $r$ satisfies $r(1)=1$, and it takes only a routine verification to show that $r$  is additive in the sense of  Proposition \ref{pro:PMequiv}(4). 
\end{proof}

The following example appears in many different contexts with mild variations, ranging from $\ell$-groups \cite{Goodearl86}  over MV-algebras \cite{Boccuto2017,KroupaMarraSC17} to integral representation theory for compact convex sets \cite{Alfsen:71}.

\begin{example}[Affine representation]\label{ex:affine}
Natural affine representations of partially ordered groups are based on affine continuous functions on their state spaces \cite[Chapter 11]{Goodearl86}. We recall  that $\Aff(A)$ denotes the linear space of~all real-valued affine continuous functions on a convex set $A$ in some linear space; see Appendix \ref{sec:coco}. The affine representation of an MV-algebra can be constructed as follows. Let $M$ be an MV-algebra, $\St M$ be its state space, and~consider the linear space $\Aff(\St M)$ that is partially ordered by the pointwise order of real functions on $\St M$. An order unit in $\Aff(\St M)$ is taken to~be the constant function $1$. Since $\St M$ is a Bauer simplex, it is a direct consequence of \cite[Theorem 11.21]{Goodearl86} that $\Aff(\St M)$ is in fact lattice ordered, hence a Riesz space. An affine representation of $M$ is the (Riesz) MV-algebra
\[
M_{\mathrm{Aff}} \coloneqq \Gamma(\Aff(\St M)).
\]
For any $a\in M$ consider a function
\[
\hat{a}(s) \coloneqq s(a), \quad s\in \St M,
\]
which is clearly affine and continuous, $\hat{a}\in M_{\mathrm{Aff}}$. Then an injective mapping $\phi\colon M\to M_{\mathrm{Aff}}$ defined by\begin{equation}\label{def:affinephi}
\phi(a)\coloneqq \hat{a}\end{equation}
is a probability map that is not an MV-homomorphism. Indeed, the equation $\phi(a\vee b)=\phi(a)\vee \phi(b)$ cannot be satisfied for all $a,b\in M$, as there are states $s\in \St M$ such that $s(a\vee b)\neq s(a\vee b)$. Observe that such states are not extreme points of $\St M$.

\end{example}

\begin{remark}
Let $M$ be an MV-algebra and $N$ be a Riesz MV-algebra \cite{DiNolaLestean14}. Boccuto et al. \cite[Definition 4.7]{Boccuto2017} adopt the following definition. A \emph{generalised state} is a map $s\colon M\to N$ satisfying $s(1)=1$ and the condition
\begin{equation*}
s(a\oplus b) = s(a)\oplus s(b), \qquad \text{for all $a,b\in M$ with $a\odot b=0$.}
\end{equation*}
On the one hand, the condition above is weaker than the one in Proposition \ref{pro:PMequiv}(4). On the other, however,  a codomain of  a probability map is not necessarily required to be a Riesz MV-algebra. Summing up, probability maps and generalised states of \cite{Boccuto2017} cannot be directly compared.
\end{remark}

\subsection{Extension to positive operators}\label{subsec:Ext}
First, we will prove a preparatory lemma, which generalises Theorem \ref{thm:groupext} about the extension of states from $M$ onto the states of the enveloping $\ell$-group $\Xi(M)$.

\begin{lemma}\label{lem:extgroup}
Let $M$ and $N$ be MV-algebras. For every $p\in \calP(M,N)$ there is precisely one unital positive group homomorphism $f_{p}\colon \Xi(M)\to \Xi(N)$ such that $f_{p}(a) = p(a)$, for all $a\in M$.
\end{lemma}
 \begin{proof}
 Let $p\colon M\to N$ be a probability map and let $A_{M}$ and $A_{N}$ be the lattice ordered monoids of \emph{good sequences} \cite[Chapter 2.2]{CignoliOttavianoMundici00} of $M$ and $N$, respectively. For any good sequence $\mathbf{a}=(a_m)_{m\in \N}\in M^{\N}$ there is some $n\in \N$ such that $a_m=0$ for all $m > n$. For any $a\in M$, let $a_0\coloneqq (a,0,0,\dots)$ be the good sequence in $A_{M}$ associated with $a$.
 
   Put 
\[
p'(a)\coloneqq p(a)_0, \quad a\in M.
\]
Observe that $p'$ is a probability map from $M$ to an isomorphic copy of $N$, the MV-algebra $\{b_0 \mid b \in N\}$.

 We define a map $\hat{p}\colon A_{M} \to A_{N}$ as 
  \[
 \hat{p}(\mathbf{a})\coloneqq \sum_{i=1}^n p'(a_i), \quad \mathbf{a}=(a_{1},\dots,a_{n},0,\dots) \in A_{M},
 \]
where  $\sum$ on the right-hand side is the sum of good sequences in $A_{N}$. 
Our goal is to prove that $\hat{p}$ is a unital monoid homomorphism. 
We get \[\hat{p}(0_0)=p'(0)=p(0)_0=0_0\] and, analogously, $\hat{p}(1_0)=1_0$. In the next step we will show that the identity
\begin{equation}\label{eq:monoidhom}
\hat{p}(\mathbf{a}+b_0) = \hat{p}(\mathbf{a}) + \hat{p}(b_0)
\end{equation}
holds true for every $\mathbf{a}=(a_{1},\dots,a_{n},0,\dots) \in A_{M}$ and  every $b\in M$. Put
\begin{align*}
	a'_1 & \coloneqq a_1\oplus b,\\
	a'_2 &\coloneqq a_2\oplus (a_1 \odot b),\\
	a'_3 &\coloneqq a_3\oplus (a_2 \odot b),\\
	&\vdots \\
	a'_n &\coloneqq a_n\oplus (a_{n-1}\odot b),\\
	a'_{n+1} &\coloneqq a_n\odot b.	
\end{align*}	
It is tedious but straightforward to show that $\ba+b_0=(a'_1,\dots,a'_{n+1},0,\dots)$. Using the definition of $\hat{p}(\ba+b_0)$, we can write
\begin{equation}\label{eq:monprob}
\hat{p}(\ba+b_0)= \sum_{i=1}^{n+1}p'(a'_i)
=p'(a_1\oplus b) + p'( a_2\oplus (a_1 \odot b)) +\dots + p'(a_n\odot b).
\end{equation}
The elements of a good sequence $\mathbf{a}$ satisfy the identity
\[
a_{i} \odot a_{i+1}=a_{i+1}, \quad i=1,\dots,n-1.
\]
Since $p'$ is a probability map we can apply Proposition \ref{pro:PMequiv}(2) to each summand in \eqref{eq:monprob} to get\begin{align*}
	& p'(a_1)+p'(b)-p'(a_1\odot b) \\	
	& + p'(a_2)+p'(a_1\odot b)-p'(\underbrace{a_2\odot a_1}_{a_2}\odot b) \\
	& + p'(a_3)+p'(a_2\odot b)-p'(\underbrace{a_3\odot a_2}_{a_3}\odot b) \\
	& + \dotsb \\
	& + p'(a_n)+p'(a_{n-1}\odot b)-p'(\underbrace{a_n\odot a_{n-1}}_{a_n}\odot b) \\
	& + p'(a_n\odot b)\\
	&= p'(a_1)+\dotsb + p'(a_n) + p'(b)\\
	&=\hat{p}(\ba)+\hat{p}(b_0).
\end{align*}
Since \eqref{eq:monoidhom} is satisfied for every $\ba\in A_{M}$ and $b\in M$, it follows that 
\[
\hat{p}(\ba + \bb) = \hat{p}(\ba) + \hat{p}(\bb)
\]
  holds for all  $\ba,\bb\in A_{M}$. Hence, $\hat{p}\colon A_{M}\to A_{N}$ is a unital monoid homomorphism. 
 
 Let $\Xi(M)$ and $\Xi(N)$ be the enveloping $\ell$-groups associated with MV-algebras $M$ and $N$, respectively. By \cite[Lemma 7.1.5]{CignoliOttavianoMundici00}, $\Xi(M)^{+}$ and $A_{M}$ are isomorphic as unital lattice ordered Abelian monoids. This means that we can identify the map $\hat{p}$ with a unital monoid homomorphism $\Xi(M)^{+}\to \Xi(N)^{+}$.

Now, every element $x\in \Xi(M)$ can be written as $y-z$, for some $y,z\in \Xi(M)^+$ (see \cite[Proposition 1.3]{Goodearl86}). Put $f_p(x)\coloneqq \hat{p}(y)-\hat{p}(z)$, for all $x\in \Xi(M)$  and some $y,z\in \Xi(M)^+$ satisfying $x=y-z$. We can routinely show that $f_p(x)$ does not depend on the choice of $y$ and $z$ and that the map $f_p\colon \Xi(M)\to \Xi(N)$ is a unital group homomorphism. Since $f_{p}(x)\in \Xi(N)^+$ for all $x\in \Xi(M)^+$, the map $f_p$ is positive and, by construction, it is the unique map satisfying $f_{p}(a) = p(a)$, for all $a\in M$.
  \end{proof}

Recall that every semisimple MV-algebra $M$ is isomorphic to the MV-subalgebra $M^{*}$ of $C(\Max M)$, where $^{*}$ is the isomorphism from Theorem~ \ref{thm:Holder}. 
 \begin{theorem}\label{thm:extoperator}
 Let $M$ and $N$ be semisimple MV-algebras and $p\colon M\to N$ be a~probability map. There exists a unique unital positive linear operator \[L_{p}\colon C_{\R}(\Max M)\to C_{\R}(\Max N)\]
 such that $p(a)^{*}=L_{p}(a^{*})$, for every $a\in M$.
 \end{theorem}
\begin{proof}
We will generalise Panti's proof of Proposition 1.1 in \cite{Panti08_InvMeasures}.
Given $p\in\calP(M,N)$ we can define a probability map $p^{*}$ from $M^{*}\subseteq C_{\R}(\Max M)$ into $N^{*}\subseteq C_{\R}(\Max N)$ by putting $p^{*}(a^{*})\coloneqq p(a)^{*}$, for all $a\in M$. We will show how to lift $p^{*}\colon M^{*}\to N^{*}$ uniquely to a unital positive linear operator $L_{p}\colon C_{\R}(\Max M)\to C_{\R}(\Max N)$. The extension consists of three steps.

(1) By Lemma \ref{lem:extgroup} we extend uniquely $p^{*}$ to a unital positive group homomorphism $f\colon G\to H$, where $G\coloneqq \Xi(M^{*})$ and $H\coloneqq \Xi(N^{*})$ are the~enveloping unital $\ell$-groups of $M^{*}$ and $N^{*}$, respectively.

(2) Since $G$ a torsion-free group, it embeds as an $\ell$-group into its divisible hull $G'$ \cite[Lemme 1.6.8, Proposition 1.6.9]{BigardKeimelWolfenstein77}. Since $M$ is semisimple its enveloping $\ell$-group is necessarily Archimedean. Thus, the elements of $G$ are  functions $a\colon \Max M\to \R$, so that $G'$ can be identified with the following Riesz space over $\Q$,
\[
G'=\{q\cdot a \mid a\in G, q\in \Q\}.
\]
We define $H'$ in a completely analogous way to $G'$. Further, one can routinely check that it is correct to define a map $f'\colon G'\to H'$ by putting $f'(q\cdot a)\coloneqq q\cdot f(a)$, for all $q\in \Q$ and $a\in G$. Clearly, $f'$ is the unique unital positive $\Q$--linear operator $G'\to H'$  extending $f$. Equipped with the sup norm $\norm{.}$, the Riesz spaces $G'$ and $H'$ become normed Riesz spaces over $\Q$. We claim that $f'$ is continuous.

To prove the claim, let $a\in G'$ and choose any strictly positive $q\in \Q$ such that $q\geq \norm{a}$. Then $\norm{\frac aq}=\frac 1q\norm{a}\leq 1$ yields $\norm{f'(\tfrac aq)}\leq 1$ by monotonicity of $f'$, which implies the inequality
\[
\norm{f'(a)}=q\norm{f'(\tfrac aq)}\leq q.
\]
Hence, $\norm{f'(a)} \leq \inf \{q\in \Q \mid q> \norm{a}\}=\norm{a}$. This means that $f'$ is continuous as this inequality holds for any $a,b\in G'$:
\[
\norm{f'(a)-f'(b)}=\norm{f'(a-b)} \leq \norm{a-b}.
\]

(3) The lattice version of Stone-Weierstrass theorem says that $G'$ is a norm-dense subspace of $C_{\R}(\Max M)$. For any $a\in C_{\R}(\Max M)$ we can find a~ sequence $(a_{n})$ in $G'$ such that $\norm{a_{n}-a}\to 0$ when $n\to \infty$. Thus, we can unambiguously define \[L(a)\coloneqq \lim_{n\to \infty} f'(a_{n}), \quad a\in C_{\R}(\Max M).\]
By density the mapping $L$ is the unique continuous linear operator extending~$f'$ to $C_{\R}(\Max M)$. Since for any $a\geq 0$ we can find a sequence of~positive elements in~$G'$ converging to $a$, the operator $L$ is positive. The~proof is finished by putting $L_{p}\coloneqq L$ and observing that it is the~desired extension of~$p$. 
\end{proof}
\begin{remark}
The conclusion of Theorem \ref{thm:extoperator} remains unchanged when $M$ is not semisimple. This is based on a certain generalisation of Proposition \ref{pro:maxsemquo} whose proof we omit, and which says that $\calP(M,N)$ and $\calP(M/\Rad M,N)$ are in bijection. Hence, one obtains a 1-1 correspondence between 
\begin{itemize}
\item probability maps $p\colon M\to N$ and
\item unital positive linear operators  $C_{\R}(\Max (M/\Rad M))\to C_{\R}(\Max N)$.
\end{itemize}

\end{remark}
In order to obtain a bijection between $\calL^{+}(C_{\R}(\Max M), C_{\R}(\Max N))$ and $\calP(M,N)$, we will assume that the codomain MV-algebra $N$ is of the form $C(Y)$ for some compact Hausdorff space~$Y$. Let $L\in \calL^{+}(C_{\R}(\Max M), C_{\R}(Y))$. Since $L$ is monotone and $C(Y)\subseteq C_{\R}(Y)$, by restricting the operator $L$ to~$M^{*}$ we obtain a probability map $M^{*}\to C(Y)$. The following result then generalises the first part of Theorem \ref{thm:intrep}.

\begin{lemma}\label{thm:statesoperators}
 Let $M$ be a semisimple MV-algebra and $Y$ be a compact Hausdorff space. Let 
 \[h\colon \calL^{+}(C_{\R}(\Max M), C_{\R}(Y))\to\calP(M,C(Y))\]
 be defined by
  \[
  h(L)(a) \coloneqq L(a^{*}),
  \]
 for all $L\in \calL^{+}(C_{\R}(\Max M), C_{\R}(Y))$ and every $a\in M$. Then $h$ is an affine isomorphism.
\end{lemma}
\begin{proof}
The sets $\calP(M,C(Y))$ and $\calP(M^{*},C(Y))$  are convex (Proposition \ref{pro:convex}), the map~$h$ is clearly affine, and it is bijective as a consequence of Theorem~\ref{thm:extoperator}. Thus, $h$ is an affine isomorphism. 
\end{proof}

We will take up a natural question of describing the extreme points of~convex set $\mathcal{P}(M,C(Y))$. This task will be presented as a direct application of Lemma~\ref{thm:statesoperators} based on the known characterisation of extreme points for the~convex set of operators $\calL^{+}(C_{\R}(\Max M), C_{\R}(Y))$; see Theorem \ref{thm:extremeoperator}. In~fact, since the map $h$ from Lemma \ref{thm:statesoperators} is an affine isomorphism,  it holds that $L\in \ext \calL^{+}(C_{\R}(\Max M), C_{\R}(Y))$ if, and only if, $h(L)\in \ext \calP(M,C(Y))$.

\begin{corollary}\label{thm:extremestate}
 Let $M$ be a semisimple MV-algebra and $Y$ be a compact Hausdorff space. For any $p\in \mathcal{P}(M,C(Y))$, these assertions are equivalent:
\begin{enumerate}
\item $p$  is an extreme point of  $\calP(M,C(Y))$.
\item $p$ is an MV-homomorphism.
\end{enumerate}
\end{corollary}

\begin{proof}
For any $p\in \calP(M,C(Y))$, define
\[
L_{p} \coloneqq h^{-1}(p),
\]
where $h$ is the affine isomorphism from Lemma \ref{thm:statesoperators}. We will show  that the~implication (1) $\Rightarrow$ (2) holds. Assume $p\in \ext \calP(M,C(Y))$. Hence, $L_{p}\in \ext \calL^{+}(C_{\R}(\Max M), C_{\R}(Y))$, which implies that $L_{p}$ is a lattice homomorphism (Theorem \ref{thm:extremeoperator}), and this means that $p\colon M\to C(Y)$ is a 
lattice homomorphism as well. Therefore $p$ satisfies the identity $p(a\oplus b)=p(a)\oplus p(b)$ and~it is an MV-homomorphism.
 
In order to prove the converse (2) $\Rightarrow $ (1), let $p\in \calP(M,C(Y))$ be an~MV-homomorphism. By Theorem \ref{thm:extremeoperator} it suffices to check that the operator $L_{p}$ is a~lattice homomorphism. Let $G\coloneqq \Xi(M^{*})$ and $H\coloneqq \Xi(C(Y))$ be the~enveloping $\ell$-groups of $M^{*}$ and $C(Y)$, respectively, and put $p^{*}(a^{*})\coloneqq p(a)$, for all $a\in M$. Then the unique extension of $p^{*}$ to a unital positive group homomorphism $f\colon G\to H$ according to Lemma \ref{lem:extgroup} is in fact an $\ell$-group homomorphism; see \cite[Chapter 7]{CignoliOttavianoMundici00}. Following the proof of Theorem \ref{thm:extoperator} we need to verify that the extended linear operator $f'\colon G'\to H'$, where $G'$ and~$H'$ are divisible hulls of $G$ and $H$, respectively, is a lattice homomorphism. Since $f'(q\cdot a)=q\cdot f(a)$ for all $q\in \Q$ and $a\in G$, we can routinely check that 
\[
f'(\tfrac am \vee \tfrac bn )=f' (\tfrac am ) \vee f'(\tfrac bn)
\]
for all $a,b\in G$ and positive integers $m,n$. The analogous identity holds for~$\wedge$. Hence, $f'\colon G'\to H'$ is a lattice homomorphism and so is its extension $L_{p}$, by~continuity of the lattice operations. Since  
\[
L_{p}\in \ext \calL^{+}(C_{\R}(\Max M), C_{\R}(Y)),
\]
 it results that $p$ is an extreme point of $\calP(M,C(Y))$ by Lemma \ref{thm:statesoperators}. 
 \end{proof}
 \begin{corollary}\label{cor:ext}
 For every semisimple MV-algebra $M$ and a compact Hausdorff space $Y$, the convex set $\calP(M,C(Y))$ has an extreme point.
 \end{corollary}
 \begin{proof}
  The map $p_{h}\colon M\to C(Y)$ defined in \eqref{eq:proeq} is an MV-homomorphism and Corollary \ref{thm:extremestate} says that $p_{h}$ is an extreme point of $\calP(M,C(Y))$.
 \end{proof}
Extreme boundary of $\calP(M,C(Y))$ can be described  in special cases.
 
 \begin{example}[Stochastic matrices]\label{ex:stmatrix}
 Recall that a \emph{stochastic matrix} is a~real square matrix with nonnegative entries satisfying the condition that the sum of every row is equal to $1$.
 Let $n \in \N$ and consider MV-algebras $M= N=[0,1]^{n}$. Clearly $N$ is of the form $C(Y)$ for $Y\coloneqq \{1,\dots,n\}$, so Corollary \ref{thm:extremestate} applies. It is even possible to fully describe the convex set $\calP([0,1]^{n},[0,1]^{n})$ as~follows. For any $p\in \calP([0,1]^{n},[0,1]^{n})$, there is a unique unital positive linear mapping $L_{p}\colon \R^{n} \to \R^{n}$ such that $L_{p}(a)=p(a)$, for all $a\coloneqq (a_{1},\dots,a_{n})\in [0,1]^{n}$.  This means that there exists a unique stochastic matrix $\mathbf{S}$ of order $n$ such that
    \begin{equation}\label{eq:stma}
    p(a)=(\mathbf{S}a^{\intercal})^{\intercal},\qquad a\in [0,1]^{n}.
    \end{equation}
  Thus, $\calP([0,1]^{n},[0,1]^{n})$ is affinely isomorphic to the $n(n-1)$-dimensional convex polytope in $\R^{n^{2}}$ of all stochastic matrices of order $n$. It is known that the extreme points of this polytope are precisely those stochastic matrices whose entries are $\{0,1\}$-valued only; see \cite[Chapter I \S 4]{Schaefer74}. Since there are $n^{n}$ of such extreme points and $n^{n}>n(n-1)$ whenever $n>1$, this means that $\calP([0,1]^{n},[0,1]^{n})$ cannot be an $n(n-1)$-dimensional (Bauer) simplex.  
 \end{example} 
 
 The previous example shows that the set of all probability maps $\calP(M,N)$ fails to be a simplex even when $M$ and $N$ possess finitely-many maximal ideals. In general, almost none of the characteristic features of the Bauer simplex $\St M$ is preserved, when passing to a more complicated collection of~maps $\calP(M,N)$.
 
 \subsection{Integral representation}\label{sec:IR}
 Every probability map from $M$ into the standard MV-algebra $[0,1]$ is Lebesgue integral \eqref{thm:int} with respect to a unique regular Borel measure over $\Max M$. This is the content of representation theorem whose proof relies on the extension of states onto positive linear functionals and the subsequent application of Riesz theorem for Banach lattice $C_{\R}(\Max M)$; see \cite{Panti08_InvMeasures,FlaminioKroupa15}. Theorem~\ref{thm:extoperator} makes it possible to pursue this path of reasoning a little further, since the~theorem says that a~probability map between semisimple MV-algebras has a unique extension to a~positive operator between the associated Banach lattices of continuous functions. Therefore it is not unreasonable to ask a~question like this: \emph{Does a probability map allow for a certain integral representation?} Theory of vector measures \cite{DiestelUhl77} provides the right framework for stating this problem precisely. In that follows only vector measures with values in spaces of the form $C_{\R}(X)$ are considered.
 
 Let $X$ and $Y$ be compact Hausdorff spaces. A \emph{vector measure} is a mapping $\vec{\mu}$ from Borel sets of $X$ into $C_{\R}(Y)$ satisfying the following condition: If~$A_{1},A_{2},\dots$ is a sequence of pairwise disjoint Borel measurable subsets of~$X$, then
 \[
 \vec{\mu} ( \bigcup_{n=1}^{\infty} A_{n} ) = \sum_{n=1}^{\infty} \vec{\mu}(A_{n}),
 \]
where the sum on the right-hand side converges in the norm topology of~Banach space $C_{\R}(Y)$. We say that  a vector measure $\vec{\mu}$ is \emph{regular} when for each Borel set $A$ and every $\epsilon >0$ there exists a compact set $K$ and an open set~ $O$ such that $K\subseteq A \subseteq O$ and $\norm{\vec{\mu}}(O\setminus K)<\epsilon$, where $\norm{\vec{\mu}}$ is the so-called \emph{semivariation} of $\vec{\mu}$; see \cite[p.2]{DiestelUhl77}. The integral of $a\in C_{\R}(X)$ with respect to a $C_{\R}(Y)$-valued regular vector measure on $X$ is an almost direct generalisation of Lebesgue integral, since the construction starts from an elementary integral of simple functions as in the classical case. Thus, for any $a\in C_{\R}(X)$, the~integral
\[
\int_{X} a \;\mathrm{d}\vec{\mu} 
\]
 has a tangible meaning and the assignment
 \begin{equation*}\label{eq:intereprevec}
 L_{\vec{\mu}} (a) \coloneqq \int_{X} a \;\mathrm{d}\vec{\mu}, \qquad a\in C_{\R}(X),
\end{equation*}
determines a bounded linear operator $L_{\vec{\mu}}\colon C_{\R}(X)\to C_{\R}(Y)$. Now, the question from a previous paragraph amounts to asking for Riesz theorem for operators. Specifically: \emph{Given a bounded linear operator $L\colon C_{\R}(X) \to C_{\R}(Y)$, is there a regular vector measure $\vec{\mu}$ on the Borel sets of $X$ with values in~$C_{\R}(Y)$ such that $L= L_{\vec{\mu}}$?} It turns out that this question has a positive answer if and only if the operator $L$ is \emph{weakly compact}, that is, $L$ maps bounded sets in $C_{\R}(X)$ to relatively weakly compact subsets of $C_{\R}(Y)$. This follows from Bartle-Dunford-Schwartz theorem; see \cite[Theorem VI.5]{DiestelUhl77}. Using this result it is very easy to find a positive operator $L$ for which no representing regular vector measure exists. Indeed, consider the identity operator $C_{\R}([0,1])\to C_{\R}([0,1])$ and note that (the weak closure of) the unit ball in $C_{\R}([0,1])$ fails to be weakly compact. Hence the identity is not a weakly compact operator. 

The above discussion implies that there are probability maps having no~integral representation in the sense introduced above. However, we can still look for conditions on $C_{\R}(X)$ and $C_{\R}(Y)$ guaranteeing that any bounded linear operator between the two spaces is weakly compact. Such sufficient conditions are provided by Grothendieck's theorem, which is the main ingredient of this theorem.
\begin{theorem}\label{thm:intvecmeas}
Let $M$ be a complete MV-algebra and $N$ be a semisimple MV-algebra with metrisable $\Max N$. For every probability map $p\colon M\to N$ there is a~$C_{\R}(\Max N)$-valued regular vector measure $\vec{\mu}$ on $\Max M$ such that
\[
p(a)^{*} = \int\limits_{\Max M} a^{*} \mathrm{d}\vec{\mu}, \qquad a\in M.
\]
\end{theorem}
\begin{proof} Every $\sigma$-complete MV-algebra is semisimple \cite[Proposition 6.6.2]{CignoliOttavianoMundici00}. A~fortiori, $M$ is semisimple. It follows from \cite{Kawano94} that the enveloping $\ell$-group $\Xi(M)$ is Dedekind complete. Since the space of extreme states of $M$ is homeomorphic to the space of extreme states of $\Xi(M)$, where the latter is extremally disconnected by \cite[Theorem 8.14]{Goodearl86}, the space $\Max M$ is extremally disconnected, too. Since $\Max N$ is metrisable, the Banach space $C_{\R}(\Max N)$ is separable. Then Grothendieck's theorem \cite[Corollary VI.12]{DiestelUhl77} says that any bounded linear operator $C_{\R}(\Max M)\to C_{\R}(\Max N)$ is weakly compact. In~particular, the positive operator $L\colon C_{\R}(\Max M)\to C_{\R}(\Max N)$ extending $p$ (Theorem \ref{thm:extoperator})  is weakly compact. Thus there exists a regular vector measure $\vec{\mu}$ (\cite[Corollary VI.14]{DiestelUhl77}) such that 
\[
L(a) = \int\limits_{\Max M} a \mathrm{d}\vec{\mu}, \qquad a\in C_{\R}(\Max M).
\]
Since $p(a)^{*}=L(a^{*})$ for all $a\in M$, this completes the proof. 
\end{proof}

\section{Dual maps}\label{sec:Adj}
An alternative representation of probability maps is studied in this section. A~possible motivation can be Example \ref{ex:stmatrix}, where a unique stochastic matrix~$\mathbf{S}$ is associated with a probability map $p\in \calP([0,1]^{n},[0,1]^{n})$, such that the identity \eqref{eq:stma} holds. The matrix is of the form \[
\mathbf{S}\coloneqq 
\begin{pmatrix}
\mathbf{s}_{1}\\
\vdots \\
  \mathbf{s}_{n}
\end{pmatrix}
\]
with each $\mathbf{s}_{i}\coloneqq (s_{i1},\dots,s_{in})\in [0,1]^{n}$ satisfying $\sum_{j=1}^{n}s_{ij}=1$. From here $n$ states $s_{1},\dotsc,s_{n}\in \St [0,1]^{n}$ can be defined:
\[
s_{i}(a) \coloneqq \sum_{j=1}^{n} s_{ij} \cdot a_{j}, \qquad a\coloneqq (a_{1},\dots,a_{j})\in [0,1]^{n}, \; i=1,\dots,n.
\]
 By \eqref{eq:stma} the relation between $p$ and $(s_{1},\dots,s_{n})$ is
\[
p(a)_{i} = s_{i}(a), \qquad \text{for all $a\in [0,1]^{n}$ and $i=1,\dots,n$.}
\]
Since $\Max\; [0,1]^{n}$ can be identified with $\{1,\dots,n\}$, one can think of $(s_{1},\dots,s_{n})$ as an $n$-tuple of states indexed by maximal ideals of $[0,1]^{n}$. The above construction can be carried out for any probability map. We spell out the details in the rest of this sectiom.

Let $M$ and $N$ be arbitrary MV-algebras and $p\in \calP(M,N)$. For any $\mathfrak{n} \in \Max N$, define a function $p_{\mathfrak{n}}\colon M\to [0,1]$ as
\begin{equation}\label{def:pn}
p_{\mathfrak{n}}(a) \coloneqq p(a)^{*}(\mathfrak{n}), \qquad a\in M.
\end{equation}
Since $p$ is a probability map and $^{*}\colon N\to C(\Max N)$ is an MV-homomorphism (Theorem \ref{thm:Holder}), the function  $p_{\mathfrak{n}}$ is a state of $M$.\begin{definition}\label{def:dual}
Let $p\in \calP(M,N)$. The \emph{dual} of $p$ is the mapping \[
p'\colon \Max N \to \St M\] 
such that $p'(\frakn)\coloneqq p_{\frakn}$, where $p_{\frakn}$ is as in \eqref{def:pn}. Put
\[
\calP'(M,N) \coloneqq \{p'\mid \text{$p'$ is the dual of some $p\in \calP(M,N)$} \}.
\]
\end{definition}
\begin{figure}[h]
\begin{center}
\begin{tikzpicture}[domain=0:4,scale=0.8]
\fill [fill=green!40,draw=green]
(-6,2) rectangle (-3,4);
\draw (-5.8,4.4) node {$M$};
\draw (2.5,5) node {$N$};
\draw[-] (-0.5,0) -- (5,0) node[right] {$\Max N$};
\draw[-] (-0.5,0) node[left] {$0$} -- (-0.5,4.2) node[left] {$1$};

\draw[color=blue,thick] plot (\x,{sin(\x r)+2});
\draw[-,dashed] (3,0) -- (3,2.14) node[right] {$\;p_{\mathfrak{n}}(a)$};
\filldraw[color=orange!20,draw=red] (-5.5,-1) -- ++(60:3) -- ++(-60:3) -- cycle ;
\draw (-5.3,1) node {$\St M$};

\draw[->] (3,-0.5) node[above] {$\mathfrak{n}$} to [out=250,in=270] (-4.1,-0.5) node[above] {$p_{\mathfrak{n}}$};
\draw[->] (-5,2.9) node[below] {$a$} to [out=70,in=90] (1.7,3.8) node[below] {$\textcolor{blue}{p(a)}$};
\draw (-1.5,5.5) node {$p$};
\draw (1.5,-2.5) node {$p'$};
\end{tikzpicture}
\end{center}
\caption{Probability map $p$ and its dual $p'$}
\end{figure}
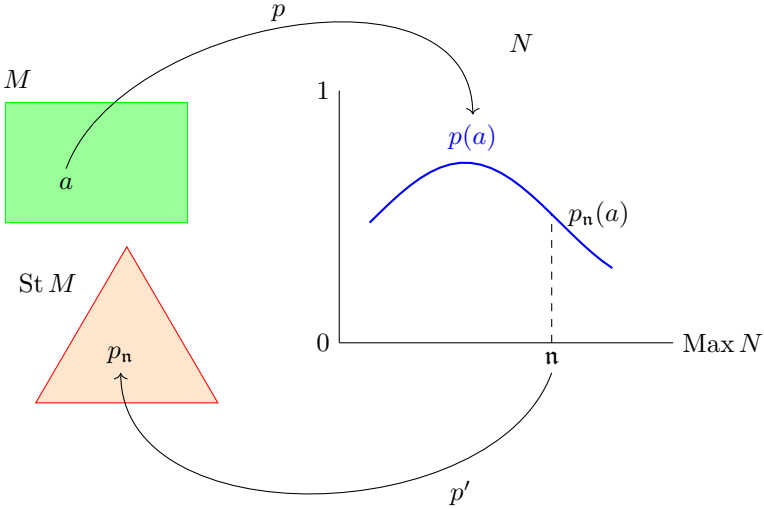
In this lemma basic properties of dual maps are collected. 
\begin{lemma}\label{lem:duals}
Let $M$ and $N$ be MV-algebras. The following hold true.
\begin{enumerate}
\item $p'$ is continuous for every  $p\in \calP(M,N)$.
\item If $N$ is semisimple, then $p\mapsto p'$ is a bijection.
\end{enumerate}
\end{lemma}

\begin{proof}
(1) Let $(\frakn_{\gamma})$ be a net in $\Max N$ converging to some $\frakn\in \Max N$. We~need to show that the limit of a net of states $(p'({\mathfrak{n}_{\gamma}}))$ in $\St M$ is $p'({\frakn})$. Let $a\in M$. Then, by \eqref{def:pn}, \eqref{def:convstates}, and continuity of $p(a)^{*} \in N^{*}$, 
\[
p'({\mathfrak{n}_{\gamma}})(a)=p_{\mathfrak{n}_{\gamma}}(a)=p(a)^{*}(\mathfrak{n}_{\gamma})\; \to\; p(a)^{*}(\mathfrak{n}) = p_{\mathfrak{n}}(a)=p'({\frakn})(a).
\]

(2) The assertion is trivially true when the cardinality of $\calP(M,N)$ is~strictly smaller than two. Let $p,q\in \calP(M,N)$ with $p\neq q$. Then there is $a\in M$ having the property $p(a)\neq q(a)$. Since $N$ is semisimple, the map $^{*}$ from $N$ onto $N^{*}$ is injective (Theorem \ref{thm:Holder}), which implies that there is some $\frakn\in \Max N$ satisfying $p(a)^{*}(\frakn) \neq q(a)^{*}(\frakn)$. By \eqref{def:pn} this inequality reads as
$p'_{\mathfrak{n}}(a)\neq q'_{\mathfrak{n}}(a)$. Hence, $p'\neq q'$. 
\end{proof}

The set of duals of probability maps can  easily be described in some special cases, which depend on the structure of $N$.
\subsection{Codomain $N=C(Y)$}
Let $M$ be an arbitrary semisimple MV-algebra and $N\coloneqq C(Y)$ for a compact Hausdorff space $Y$. Since $\Max C(Y)$ is homeomorphic to $Y$ (see \cite[Theorem 4.16(iv)]{Mundici11}), we will tacitly identify the elements of $Y$ with maximal ideals in ~$\Max C(Y)$. Let $f\colon Y\to \St M$ be any continuous map. Then one can routinely check that the map $p\colon M\to C(Y)$ defined by $p(a)(y)\coloneqq f(y)(a)$, for all $a\in M$ and all $y\in Y$, is a probability map and its dual is $p'=f$.   Thus, by Lemma \ref{lem:duals}, the dual of $\calP(M,C(Y))$ is
\begin{equation}\label{eq:duals}
\calP'(M,C(Y)) = \{f \mid \text{$f\colon Y \to \St M$ is continuous}\}.
\end{equation}
The set $\calP'(M,C(Y)) $ is clearly a nonempty convex subset of the real linear space $(\St M)^{Y}$. In particular, $\calP'(M,C(Y))$ is affinely isomorphic to the convex set $\calP(M,C(Y))$ under the bijection $p\mapsto p'$, which is a direct consequence of Lemma \ref{lem:duals} and the definition of $p'$.

\begin{theorem}\label{thm:dualschar}
Let $M$ be a semisimple MV-algebra and $Y$ be a compact Hausdorff space. Then $\ext \calP(M,C(Y))\neq \emptyset$ and the following assertions are equivalent for every $p\in \calP(M,C(Y))$.
\begin{enumerate}
\item $p\in \ext \calP(M,C(Y))$.
\item $p'\in \ext \calP'(M,C(Y))$.
\item $p$ is an MV-homomorphism.
\item $p'[Y]\subseteq \ext \St M$.
\item There is a unique continuous map $f\colon Y \to \Max M$ satisfying the equality $p(a)=a^{*}\circ f$, for all $a\in M$.
\end{enumerate}
\end{theorem}
\begin{proof}
By Corollary \ref{cor:ext}, $\ext \calP(M,C(Y))\neq \emptyset$.
 Since the map $p\mapsto p'$ is an~affine isomorphism of $\calP(M,C(Y))$ onto $\calP'(M,C(Y))$, claims (1) and (2) are equivalent. The equivalence of (1) and (3) follows from Corollary \ref{thm:extremestate}.  We~prove that (3) implies (4). Let $p\colon M\to C(Y)$ be an MV-homomorphism. By \eqref{eq:stspacemax} we need only show that $p'(y)$ is an MV-homomorphism $M\to [0,1]$, for all $y\in Y$. But this is a straightforward consequence of the assumption and the definition of $p'$. Conversely, let (4) be true and we want to conclude that~$p$ is an MV-homomorphism. For all $a,b\in M$ and every $y \in Y$, the~definition of $p'$ yields
\[
p(a\oplus b)(y)=p'_{y}(a\oplus b) =p'_{y}(a) \oplus p'_{y}(b) = p(a)(y) \oplus p(b)(y).
\]
Since $p$ is a probability map and satisfies $p(a\oplus b)=p(a)\oplus p(b)$, it is an~MV-homomorphism. The equivalence of (4) and (5) follows from \eqref{eq:duals} and the~fact that $\ext \St M$ is homeomorphic to $\Max M$. 
\end{proof}

\subsection{Codomain $N=$ Affine representation of $M$}
Assume that $M$ is a semisimple MV-algebra and $N\coloneqq M_{\mathrm{Aff}}$, where $M_{\mathrm{Aff}}$ is as in Example \ref{ex:affine}. Since $M_{\mathrm{Aff}}$ is a convex subset of $\R^{\St M}$, the set $\calP(M,M_{\mathrm{Aff}})$ is convex as well. As $M_{\mathrm{Aff}}$ is a separating subalgebra of $C(\St M)$, the space $\Max M_{\mathrm{Aff}}$ is homeomorphic to $\St M$, so we may consider dual maps as mappings $\St M \to \St M$. For any $p\in \calP(M,M_{\mathrm{Aff}})$, it is easy to see that the dual map $p'\colon \St M \to \St M$ is  continuous (Lemma \ref{lem:duals})  and affine. Conversely, let a map $f\colon  \St M \to \St M$ be affine continuous. Then the map $p_{f}\colon M\to M_{\mathrm{Aff}}$ defined by $p_{f}(a)=\hat{a}\circ f$, for all $a\in M$, is a probability map satisfying $p_{f}'=f$. We can conclude
\begin{equation*}
\calP'(M,M_{\mathrm{Aff}}) = \{f \mid \text{$f\colon \St M \to \St M$ is affine continuous}\},
\end{equation*}
and the convex sets $\calP(M,M_{\mathrm{Aff}})$ and $\calP'(M,M_{\mathrm{Aff}})$ are affinely isomorphic.

\begin{remark}
The reader familiar with the construction of state space functor in \cite[Chapter 6]{Goodearl86} might point out an alternative notion of dual map. Namely consider a~probability map $p\colon M\to N$ and a map $p^{\star}\colon \St N\to \St M$ defined by precomposing with $p$. Specifically,
\[
p^{\star}(s) \coloneqq s\circ p, \quad \text{for all $s\in \St N$.}
\]
Since $\Max N$ and $\ext \St N$ are homeomorphic under $\frakn \mapsto s_{\frakn}$, where $s_{\frakn}$ is as~in~\eqref{def:statemorph}, we get
\[
p^{\star}(s_{\frakn})(a) = s_{\frakn}(p(a)) =p(a)^{*}(\frakn)=p'_{\frakn}(a),
\]
for all $a\in M$. Hence, $p^{\star}(s_{\frakn})=p'(\frakn)$, for any $\frakn\in \Max N$. Since $p'$ is a~continuous function that coincides with $p^{\star}$ over $\ext \St N$ and $p^{\star}$ is clearly an~affine continuous functions, then Lemma \ref{lem:affext} says that $p^{\star}$ is the unique extension of~$p'$ over the entire state space $\St N$. Therefore the map $p'$, which is defined on a less complicated domain than $p^{\star}$, contains the same amount of~information about $p$ as $p^{\star}$.

\end{remark}

\section{Further research}\label{sec:conc}
It is clear from Section \ref{sec:IR} that one cannot obtain a general integral representation theorem for any probability map $M\to N$, even though both $M$ and~ $N$ are relatively well-behaved. However, it makes sense to ask whether the~conclusion of Theorem \ref{thm:intvecmeas} can be extended beyond the assumptions on~$M$ and~$N$ stated in that theorem. Specifically, \emph{find the conditions on a~probability map $p\colon M\to N$ which guarantee that $p$ extends to a weakly compact positive operator.}

The material presented in Section \ref{sec:Adj} leads to the following question. \emph{Given MV-algebras $M$ and $N$, can we characterise duals of probability maps $M\to N$ as continuous functions $\Max N\to \St M$ satisfying special properties?} It is known that this problem has a solution only in special cases, such as when $M$, $N$ are finitely presented algebras and probability maps are homomorphisms $M\to N$. Indeed, it follows from   \cite{MarraSpada12} and Theorem \ref{thm:dualschar} that the duals of such homomorphisms are precisely continuous piecewise linear maps $\Max N\to \Max M$ with integer coefficients. Therefore, one might at least ask if  the duals of probability maps $M\to N$ between two finitely presented MV-algebras can be characterised as a special class of continuous functions $\Max N\to \St M$.

\appendix

\section{Compact convex sets}\label{sec:coco}
The book \cite{Alfsen:71} is a classical reference concerning compact convex sets in infinite-dimensional spaces and integral representation theorems. Since the importance of this theory for classifying state spaces of partially ordered Abelian groups cannot be overestimated, the interested reader might also want to~check \cite[Chapter 5]{Goodearl86}. See \cite{Holmes75} for all the unexplained terminology and results related to functional analysis and basic convexity.

In this section we assume that $K$ and $L$ are real locally convex Hausdorff spaces. Let $A\subseteq K$ and $B\subseteq L$ be convex sets. A map $f\colon A\to B$ is \emph{affine} when $f(\alpha x + (1-\alpha)y)=\alpha f(x) + (1-\alpha)f(y)$, for every $x,y \in A$ and all $\alpha \in [0,1]$. If $f$ is affine and bijective, then $f^{-1}$ is affine as well, and we say that $f$ is an \emph{affine isomorphism}. Assume that the convex sets $A$ and $B$ are in addition compact. Then $f\colon A\to B$ is an \emph{affine homeomorphism} if $f$ is an~affine isomorphism and a homeomorphism.

If $A$ is a nonempty compact subset of $K$, then it has at least one extreme point. The set $\ext A$ of all extreme points of $A$ is called an \emph{extreme boundary}. Every compact convex set is generated by its extreme boundary in~the following sense.

\begin{KM}
Let $A\subseteq K$ be a nonempty compact convex set. Then $A$ is the closed convex hull of its extreme boundary,
\[
A = \overline{\conv} \ext A.
\]
\end{KM}

\begin{example}[Probability measures]\label{ex:PMs}
Let $X$ be a compact Hausdorff topological space. By $\Delta(X)$ we denote the convex set of all regular Borel probability measures on $X$.  The set $\Delta(X)$ becomes a compact subset of the locally convex space of all regular Borel probability measures on $X$, which is endowed with the weak$^{*}$-topology. In this topology, a net $(\mu_{\gamma})$ in $\Delta(X)$ converges to~$\mu\in \Delta(X)$  if and only if $\lim_{\gamma}\int_X f\;\mathrm{d}\mu_{\gamma}=\int_{X} f\;\mathrm{d}\mu$, for every continuous function $f\colon X\to \R$. Given $x\in X$, let $\delta_{x}$ be the corresponding Dirac measure. Then
\[
\ext \Delta(X) = \{\delta_{x} \mid x\in X\}
\]
and the map $x\in X \mapsto \delta_{x}\in \ext \Delta(X)$ is in fact a homeomorphism. By~Krein-Milman theorem the set of all convex combinations of Dirac measures is dense in $\Delta(X)$.
\end{example}

There are several caveats that concern usability of Krein-Milman theorem. In practical situations one is mostly interested in a workable description of the extreme boundary $\ext A$ and its measure-topological properties (compactness or measurability). Moreover, easy examples from a finite-dimensional space show that even if $A$ is a convex polytope (so necessarily $\ext A$ is finite and $A=\conv \ext A$ by Minkowski theorem), there can be many convex combinations representing given $x\in A$. The compact convex sets, whose properties are on the contrary very convenient, are particular infinite-dimensional simplices called Bauer simplices. We will spell out the definition of Bauer simplex using one of their equivalent characterisations from \cite[Theorem II.4.1]{Alfsen:71}. To~this end, we define $\Aff(A)$ to be the linear space of all real-valued affine continuous functions on a convex set $A\subseteq K$.

Let $\emptyset \neq A\subseteq K$ be compact convex. We say that $x\in A$ is the \emph{barycenter} of $\mu\in \Delta(A)$ if this condition holds true:
\[
f(x) = \int_{A} f\;\mathrm{d}\mu, \qquad \text{for all $f\in \Aff(A)$.}
\]
Clearly, every point $x\in A$ is the barycenter of at least one measure $\mu\in \Delta(A)$ since it suffices to take $\mu\coloneqq \delta_{x}$. Conversely, it is not difficult to show that, for every $\mu\in \Delta(A)$, there is precisely one point $x\in A$ that is the barycenter of $\mu$. The~\emph{barycenter mapping} is  \[r_{A}\colon \Delta(A) \to A\] assigning to every $\mu\in \Delta(A)$ its unique barycenter $r_{A}(\mu)\in A$. The barycenter mapping is surjective (since $x=r_{A}(\delta_{x})$ for any $x\in A$), affine, and continuous.

\begin{definition}
Let $A$ be a compact convex subset of $K$. Then $A$ is called a~\emph{Bauer simplex} if, for every $x\in A$, there is a unique $\mu\in\Delta(A)$ supported by $\overline{\ext A}$, and such that $x$ is the barycenter of $\mu$, that is, $r_{A}(\mu)=x$.
\end{definition}
\noindent
The extreme boundary of a Bauer simplex is necessarily closed, $\overline{\ext A} =\ext A$. Any finite-dimensional simplex is a Bauer simplex and so is the set of~probability measures $\Delta(X)$ from Example \ref{ex:PMs}. In the paper we make use of~the~ following fact, an abstract variant of the~Dirichlet problem of~the~extreme boundary. Although the statement is implicit in many textbooks, we provide its standalone proof.

\begin{lemma}\label{lem:affext}
Let $K$ and $L$ be real locally convex Hausdorff spaces, $A\subseteq K$ be a Bauer simplex, and $B\subseteq L$ be compact convex. If $f\colon \ext A \to B$ is continuous, then there is a unique affine continuous function $\bar{f}\colon A \to B$ such that $\bar{f}(x)=f(x)$, for all $x\in \ext A$.
\end{lemma}
\begin{proof}
By \cite[Corollary II.4.2]{Alfsen:71}, every Bauer simplex $A$ is affinely homeomorphic to $\Delta(\ext A)$ via   the barycenter mapping $r_{A}\colon \Delta(\ext A) \to A$. Let $\mu_{x} \coloneqq r^{-1}_{A}(x)$ be the unique representing probability measure for $x\in A$. Let $f\colon \ext A \to B$ be a continuous map. Define $f_{\sharp}\mu_{x} \coloneqq \mu_{x}\circ f^{-1}$ and~notice that $f_{\sharp}\mu_{x} \in \Delta(B)$. Hence, it is sensible to define
\[
\bar{f}(x) \coloneqq r_{B}(f_{\sharp}\mu_{x}), \qquad x\in A,
\]
where $r_{B}\colon \Delta(B)\to B$ is the barycenter mapping on $B$. We will check that $\bar{f}$ is the desired extension of $f$.

If $x\in \ext A$, then $\mu_{x}=\delta_{x}\in \ext \Delta(\ext A)$, and
\[
\bar{f}(x)=r_{B}(f_{\sharp}\delta_{x})=r_{B}(\delta_{f(x)})=f(x).
\]

The mapping $\bar{f}$ is affine. To show this, let $x,y\in A$, $\alpha \in [0,1]$, and put $z\coloneqq \alpha x + (1-\alpha)y$. Then, by affinity of barycenter maps, we get
\begin{align*}
\bar{f}(z)& =r_{B}(f_{\sharp}\mu_{z})=r_{B}(f_{\sharp}(\alpha \mu_{x} + (1-\alpha)\mu_{y}) ) \\
&=\alpha r_{B}(f_{\sharp}\mu_{x}) + (1-\alpha)r_{B}(f_{\sharp}\mu_{y})=\alpha \bar{f}(x) + (1-\alpha)\bar{f}(y).
\end{align*}

To prove continuity of $\bar{f}$, consider a net $(x_{\gamma})$ converging to some $x$ in~$A$. Since $r_{A}$ is a homeomorphism, this gives $\mu_{x_{\gamma}} \to \mu_{x}$ in $\Delta(\ext A)$. Hence, $f_{\sharp}\mu_{x_{\gamma}} \to f_{\sharp}\mu_{x}$ in $\Delta(B)$. By continuity of $r_{B}$ this implies 
\[
\bar{f}(x_{\gamma})=r_{B}(f_{\sharp}\mu_{x_{\gamma}}) \to r_{B}( f_{\sharp}\mu_{x})=\bar{f}(x).
\]

Let $g\colon A\to B$ be an affine continuous map such that $g(x)=f(x)$, for all $x\in \ext A$. We want to conclude that $g(x)=\bar{f}(x)$, for every $x\in A$. Proceeding by cases, first assume $x\in \ext A$. By the assumption, $g(x)=f(x)=\bar{f}(x)$. Now, let $x\coloneqq \sum_{i=1}^{n}\alpha_{i} x_{i}$, for some positive integer $n$, where $x_{i}\in \ext A$ and~$\alpha_{i}\geq 0$ with $\sum_{i=1}^{n}\alpha_{i}=1$. By affinity of $f$ and $g$ and by the assumption,
\[
\bar{f}(x)=\sum_{i=1}^{n}\alpha_{i} \bar{f}(x_{i})=\sum_{i=1}^{n}\alpha_{i} g(x_{i})=g(x).
\]
 Krein-Milman theorem says that the only remaining case is when $x\in A$ is the limit of a net $(x_{\gamma})$ in $A$, where each $x_{\gamma}$ is a finite convex combination of~extreme points. Then, by continuity of $\bar{f}$ and $g$ together with what we have proved above,
\[
\bar{f}(x)=\lim_{\gamma} \bar{f}(x_{\gamma})=\lim_{\gamma} g(x_{\gamma})=g(x).
\]
This finishes the proof. 
\end{proof}

\subsection*{Acknowledgment}

 The author is grateful to Prof. Vincenzo Marra (University of Milan) for many suggestions and inspiring comments.


\end{document}